\documentclass{article}
\usepackage{amssymb,amsmath,amsthm,latexsym,enumerate,amsrefs}
\usepackage{pdfsync}

\newtheorem{proposition}{Proposition}[section]

\newtheorem{corollary}{Corollary}[section]
\newtheorem{lemma}{Lemma}[section]

\newcommand{\Ric}{\text{Ric}}

\begin{document}

\title{The index of symmetry for a left-invariant metric on a solvable three-dimensional Lie group}

\author{Robert D. May}

\date{}

\maketitle

\textbf{Abstract:}  We find the index of symmetry for all solvable three-dimensional Lie groups with a left-invariant metric.  When combined with the work of Reggiani on unimodular three-dimensional Lie groups, the index of symmetry is then known for all three-dimensional Lie groups with a left-invariant metric.  Similar to the unimodular case, for every solvable three-dimensional Lie algebra there is at least one left-invariant metric for which the index of symmetry is positive.  Also, the index is never equal to 2. 
\\

\\


\section{Introduction}

The \textit{index of symmetry}, $i(G)$, for a Lie group $G$ with identity element $e$ and a left-invariant metric $g$ can be defined as follows: 
Define a vector field $X$ on $G$ to be \textit{parallel at the identit}y (or parallel at $e$) if $\nabla_Y(X)$ vanishes at $e$ for every vector field $Y$ (that is, $\nabla_Y(X)\mid_e = 0$). Then define a subspace $S_e \in \mathfrak{g}$ where $\mathfrak{g}$ is the Lie algebra for $G$ considered as the tangent space to $G$ at $e$:  $S_e$ consists of all vectors $x \in \mathfrak{g}$ such that $x = X\mid_e$ for a vector field $X$ on $G$ which is a Killing field for the metric $g$ and which is parallel at $e$.  Then the index of symmetry for $G$ is the dimension of $S_e$:  $i(G) = i(G,g) = \text{dim}(S_e)$.  

Reggiani in \cite{Reg} has computed this index of symmetry for all three-dimensional Lie groups with a left-invariant metric which are \textit{unimodular} (that is, whose Lie algebras are unimodular).  In this paper we compute the index of symmetry for all three-dimensional Lie groups with a left-invariant metric which are \textit{solvable} (that is, whose Lie algebras are solvable).  Since any three-dimensional Lie group is either unimodular or solvable (or both), the index of symmetry is determined for any three-dimensional Lie group with a left-invariant metric.

In section 2 we give a table listing  $i(G)$ for all cases where the index of symmetry is positive. We note that, analogous to the results of Reggiani \cite{Reg} for the unimodular Lie algebras, for every solvable three-dimensional Lie algebra there is at least one metric for which the index of symmetry is positive.  Also $i(G)$ is never equal to $2$.

In section 3 we review the classification of solvable metric Lie algebras as given by May and Wears in \cite{MW} in terms of standard bases and defining triples $(a,b,c)$.  We note which  underlying Lie algebras correspond to  which triples.

In section 4 we give matrix representations for the solvable metric Lie algebras and the associated simply-connected Lie groups $G$.  We find global coordinates $\{x_1,x_2,x_3\}$ for each such $G$ with associated global coordinate frame $\{\partial/\partial x_1,\partial/\partial x_2,\partial/\partial x_3\}$ for the tangent bundle of $G$.  We also find a global orthonormal frame consisting of left-invariant vector fields $L_i$ and a global frame of right-invariant vector fields $R_i$ (which are Killing fields for the given metric on $G$). Finally, we find the relations between these three frames for the tangent bundle.

In section 5 we review some properties of isometries of $G$ and the related Killing vector fields.  We obtain a condition for a vector field $X$ to be a Killing field in terms of the left-invariant frame vectors $L_i$. We define a matrix $P(X)$ for any vector field $X$ such that the vanishing of $P(X)$ is equivalent to the condition that $X$ is parallel at the identity.  We then compute the values of $P(X)$ when $X$ is one of the right invariant vector fields $R_i$.  Finally, we obtain a formula for $P(X)$ when $X$ is a Killing field vanishing at $e=(0,0,0)$. 

In section 6 we review the index of symmetry $i(G)$.  We then define a subspace $\bar{S}_e \subseteq S_e$ and a \textit{subindex} $\bar{i}(G) = \text{dim}(\bar{S}_e)$. 
$\bar{S}_e$ and $\bar{i}_e$ involve only the right-invariant Killing fields. We give a condition under which $i(G) = \bar{i}(G)$. Then using the formulas for $P(R_j)$ from section 5 we calculate the subindex $\bar{i}(G)$ for all cases. 

In section 7 we study certain restrictions on the contributions that Killing fields $X$ vanishing at $e$ can make to the index of symmetry. We define the Ricci matrix $R$ to be the Ricci tensor for $g$ at $e$ expressed in terms of the standard basis $\{e_i\}$. We  obtain a formula for $P(X)$ in terms of the one-parameter group of isometries fixing $e$ corresponding to $X$ and indicate how the structure of $R$ can influence $P(X)$.

In section 8 we define a Lie algebra (and associated Lie group $G$) to be \textit{generic} if the diagonal entries in the diagonalized Ricci matrix are all distinct.  Our main result is that for generic $G$ the subindex and index agree:  $i(G) = \bar{i}(G)$. Then using the results of section 6 we have the formula for $i(G)$ in all but the nongeneric cases.  We catalog all the nongeneric metric Lie algebras.

In sections 9, 10, and 11 we calculate $i(G)$ for the nongeneric metric Lie algebras.

In section 12 we sketch the computational proofs of some of our stated results. 

We use the summation convention throughout: unless specified otherwise, any repeated index implies summation over that index.  Also, we use $e$ for the identity in a Lie group $G$, so $e=(0,0,0)$ for coordinates centered at the identity.  We then write $X\lvert_e$ for the value of a vector field $X$ on $G$ at the identity $e$. So $X\lvert_e \in \mathfrak{g}$ when we identify the Lie algebra $\mathfrak{g}$ with the tangent space to $G$ at $e$.  
  
\section{Table of results.}

Table 1 below summarizes our results for the index of symmetry $i(G)$ for a simply-connected Lie group $G$ with a left-invariant metric in terms of the correponding metric Lie algebra $\mathfrak{g}$.  We list all cases for which the index of symmetry is positive.  In all other cases the index is $0$.  Notice that for every solvable algebra there is at least one metric for which the index of symmetry is positive.

\begin{table} [h]
	\begin{center}
		\begin{tabular}{|l|c|c|c|}
			\hline
			$\mathfrak{g}$  &  Triple $(a,b,c)$  &  $S_e$  &  $i(G)$  \\ \hline 
			$\mathfrak{g}(K)$, $K>1$  &  $(a,b,b)$, $a>0$
			& span$(e_1)$  &  1  \\ \hline 	 
			$\mathfrak{g}(1)$  &  $(a,b,b),\,a>0$  &  
			$\mathfrak{g}$  &  3  \\ \hline
			$\mathfrak{g}(1)$  &  $(a,b,c),\,a>0,\,b\neq c $  &  span$(\alpha_2e_2+\alpha_3 e_3)$  &  1  \\ \hline
			$\mathfrak{g}(K)$, $0<K<1$  &  $(a,b,b)$, $a>0$
			& span$(e_1)$  &  1  \\ \hline $\mathfrak{g}(K)$, $0<K<1$  &  $(a,a+a\sqrt{1-K},a-a\sqrt{1-K})$, $a>0$
			& span$(e_2-e_3)$  &  1  \\ \hline
			$\mathfrak{g}(0)$  &  $(a,2a,0)$, $a>0$
			& span$(e_2-e_3)$  &  1  \\ \hline
			$\mathfrak{g}(K)$, $K<0$  &  $(a,b, -b)$, $a>0$
			& $\mathfrak{g}$  &  3  \\ \hline
			$\mathfrak{g}(K)$, $K<0$  &  $(a,a+a\sqrt{1-K},a-a\sqrt{1-K})$, $a>0$
			& span$(e_2-e_3)$  &  1  \\ \hline 
			$\mathfrak{g}_I$  & $(a,0,0)$, $a>0$
			& $\mathfrak{g}$  &  3  \\ \hline
			$\mathfrak{e}(1,1)$ &  $(0,b,b)$, $b>0$
			& span$(e_1)$  &  1  \\ \hline
			Heisenberg algebra  &  $(0,b,0)$, $b>0$
			& span$(e_3)$  &  1  \\ \hline
			$\mathfrak{e}(2)$  &  $(0,b,-b)$, $b>0$
			& $\mathfrak{g}$  &  3  \\ \hline
			Abelian algebra   &  (0,0,0) & $\mathfrak{g}$  
			&  3  \\  \hline
		\end{tabular}
		\caption{Index of symmetry $i(G)$} 	
	\end{center}
\end{table}

 \section{Classification of solvable three-dimensional Lie algebras} 

Two metric Lie algebras $(\mathfrak{g}_1,Q_1) , (\mathfrak{g}_2,Q_2)$ are equivalent if there exists a Lie algebra isomorphism $f:\mathfrak{g}_1 \rightarrow \mathfrak{g}_2$ which also preserves the metric, $f^*(Q_2)=Q_1$. By the results of May and Wears \cite{MW}, for any three-dimensional solvable metric Lie algebra $(\mathfrak{g},Q)$, there exists an orthonormal basis ${\mathbf{e}} = \{ e_1 ,e_2 ,e_3 \} $ and a matrix $A = \left[ {\begin{array}{*{20}c}
	a & c  \\
	b & a  \\
	\end{array} } \right]$ with $a\geq0,b\geq\mid c\mid$, such that
$[e_1,e_2 ] = ae_2  + be_3$ and $[e_1,e_3 ] = ce_2  + ae_3$. We say that the metric Lie algebra $(\mathfrak{g},Q)$ defines the triple $(a,b,c)$.  Equivalence classes of solvable metric Lie algebras correspond one to one with such matrices $A$ (or such triples $(a,b,c)$).  These solvable metric Lie algebras $(\mathfrak{g},Q)$ correspond one to one with simply-connected three-dimensional solvable Lie groups $G$ with a left-invariant metric.

The underlying Lie algebra $\mathfrak{g}$ corresponding to a given matrix $A = \left[ {\begin{array}{*{20}c}
	a & c  \\
	b & a  \\
	\end{array} } \right]$ with $a\geq0,b\geq\mid c\mid$ can be described as follows: 

First, $\mathfrak{g}$ is unimodular if and only if $a = 0$.

If $a=0, \, b \geq c >0$, then $\mathfrak{g} = \mathfrak{e}(1,1)$, the Lie algebra for the Lie group $E(1,1)$ of rigid motions of Minkowski 2-space.  (The simply-connected Lie group $G$ is the universal cover of the connected component of the identity of $E(1,1)$.)

If $a=0, \, c<0$, then $\mathfrak{g} = \mathfrak{e}(2)$, the Lie algebra for the Lie group $E(2)$ of rigid motions of Euclidean 2-space.  (The simply-connected Lie group $G$ is the universal cover of the connected component of the identity of $E(2)$.)

If $a=0 , \, b>0, \, c=0$, then  $\mathfrak{g}$ is the Heisenberg algebra and $G$ is the associated Heisenberg group. 
 
If $a=b=c=0$, then the algebra is abelian.
  
Next, consider the non-unimodular cases where $a > 0\,,\,b \geqslant \left| c \right|$.  If $b=c=0$ , then $\mathfrak{g}$ is the special case of a (nonunimodular) Lie algebra with a basis $\{ b_1 ,b_2 ,b_3 \} $ for which $[b_1 ,b_2 ] = b_2 \,,\,[b_1 ,b_3 ] = b_3 \,,\,[b_2 ,b_3 ] = 0$.  We will denote this Lie algebra by $\mathfrak{g}_I $,  

For all other cases when $a > 0 , \, b > 0$, define 
\[K = K(\mathfrak{g},Q) = \frac{{bc}}{{a^2 }}.\]  

The following result was proved by May and Wears in \cite{MW} (and was observed by Milnor in \cite{Mil}).   
\begin{lemma}  For $i=1,2$, let $(\mathfrak{g}_i,Q_i)$ be non-unimodular metric Lie algebras defining triples $(a_i,b_i,c_i)$ with $a_i>0,b_i>0$.  Then $\mathfrak{g}_1$ and $\mathfrak{g}_2$ are isomorphic as Lie algebras if and only if $K(\mathfrak{g}_1,Q_1 ) = K(\mathfrak{g}_2 , Q_2)$.  \end{lemma}

Thus for each real $K$ there is a Lie algebra $\mathfrak{g}(K)$ such that the metric Lie algebras $(\mathfrak{g},Q)$ with $\mathfrak{g}$ isomorphic to $\mathfrak{g}(K)$ are exactly those with $K(\mathfrak{g},Q) = K$. 
Then for metric Lie algebras defining triples $(a,b,c) $  with $a>0$ and $\frac{{bc}}{{a^2 }}= K$ the underlying Lie algebra is the (non-unimodular) algebra $\mathfrak{g}(K)$ when $K \ne 0$.  When $K=0$, the Lie algebra is still $\mathfrak{g}(0)$ as long as $b>0$ .  When the defined triple is $(a,0,0) , a>0 $, the underlying Lie algebra is the special Lie algebra $\mathfrak{g}_I $.

\section{Matrix representations and global coordinates}

In this section we consider matrix representations for the solvable Lie algebras and then construct associated simply-connected matrix Lie groups for these Lie algebras.  Let $A = \left[ {\begin{array}{*{20}c}   a & c  \\   b & d  \\
	\end{array} } \right]$ be any non-zero two-by-two real matrix and write $\mathfrak{g}\left( A \right)$ for the subset of $M_3 (\mathbb{R})$ consisting of matrices of the form
\[L(\mathbf{x}) = L(x_1,x_2,x_3) = \left[ {\begin{array}{*{20}c}
	{x_1 A} & {\begin{array}{*{20}c}   x_2  \\   x_3  \\
		\end{array} }  \\
	{\begin{array}{*{20}c} 
		0 & 0  \\
		\end{array} } & 0  \\
	\end{array} } \right].\]
Computing brackets in $M_3 (\mathbb{R})$ gives 
$ \left[ L(\mathbf{x}),L(\mathbf{y}) \right] = L(\mathbf{z})
$
where $z_1  = 0$ and
\[\left[ {\begin{array}{*{20}c}
	{z_2 }  \\
	{z_3 }  \\
	\end{array} } \right] = x_1 A\left[ {\begin{array}{*{20}c}
	{y_2 }  \\
	{y_3 }  \\
	\end{array} } \right] - y_1 A\left[ {\begin{array}{*{20}c}
	{x_2 }  \\
	{x_3 }  \\
	\end{array} } \right] = A\left[ {\begin{array}{*{20}c}
	{x_1 y_2  - y_1 x_2 }  \\
	{x_1 y_3  - y_1 x_3 }  \\
	
	\end{array} } \right].\]
Then $\mathfrak{g}\left( A \right)$ is a three dimensional Lie subalgebra of $M_3 \left( \mathbb{R} \right)$ with a basis 
\[e_1  = L(1,0,0), e_2  = L(0,1,0), e_3  = L(0,0,1). \]  Evidently the derived algebra $\mathfrak{g}(A)'$ is contained in the span of $e_2 $ and $e_3 $ and is therefore of dimension at most two.  Compute \[\begin{gathered}
\left[ {e_1 ,e_2 } \right] = L(0,a,b) = ae_2  + be_3  \hfill \\
\left[ {e_1 ,e_3 } \right] = L(0,c,d) = ce_2  + de_3  \hfill \\
\left[ {e_2 ,e_3 } \right] = L(0,0,0) = 0 \hfill \\ 
\end{gathered} \]

If we put a metric on $\mathfrak{g}(A)$ by declaring $\left\{ {e_1 ,e_2 ,e_3 } \right\}$ to be an orthonormal frame and then choose $A$ with $a = d \geqslant 0\,,\,b \geqslant \left| c \right|$, then $\mathfrak{g}\left( A \right)$ is a matrix representation for the solvable non-abelian metric Lie algebra corresponding to the triple $\left( {a,b,c} \right) $.

We now construct a (simply-connected) matrix Lie group $G(A)$ associated with $\mathfrak{g}(A)$.  Given the matrix $A$ and a real number $t$, define a two by two matrix $M(t) = M(A,t) = e^{At} $.  Note that $M(t_1 )M\left( {t_2 } \right) = M(t_1  + t_2 )$ and that $M\left( t \right)^{ - 1}  = M( - t)$.  Define a subset $G(A)$ of $GL(3,\mathbb{R})$ to be all matrices of the form 
\[G_A(\mathbf{x}) = G_A (x_1,x_2,x_3) = \left[ {\begin{array}{*{20}c}
	{M(x_1)} & {\begin{array}{*{20}c}
		x_2  \\
		x_3  \\
		\end{array} }  \\
	{\begin{array}{*{20}c}
		0 & 0  \\
		\end{array} } & 1  \\
	\end{array} } \right].\]
Then 
\[G_A (\mathbf{x})G_A (\mathbf{y}) = G_A (\mathbf{z})\] where $z_1  = x_1  + y_1 $ while
\[\left[ {\begin{array}{*{20}c}
	{z_2 }  \\
	{z_3 }  \\
	\end{array} } \right] = \left[ {\begin{array}{*{20}c}
	{x_2 }  \\
	{x_3 }  \\
	\end{array} } \right] + M(x_1 )\left[ {\begin{array}{*{20}c}
	{y_2 }  \\
	{y_3 }  \\
	\end{array} } \right].\]
Also, 
\[G_A (x_1,x_2,x_3)^{ - 1}  = G_A (\bar{x}_1,\bar{x}_2,\bar{x}_3)\] 
where $\bar{x}_1 =  - x_1$ and 
\[\left[ {\begin{array}{*{20}c}
	{\bar{x}_2}  \\
	{\bar{x}_3}  \\
	\end{array} } \right] =  - M(x_1)^{ - 1} \left[ {\begin{array}{*{20}c}
	x_2  \\
	x_3  \\
	\end{array} } \right] =  - M( - x_1)\left[ {\begin{array}{*{20}c}
	x_2  \\
	x_3  \\
	\end{array} } \right].\]
So $G(A)$ is a subgroup of $GL(3,\mathbb{R})$ as desired.
 
For future computations it will be useful to have an explicit formula for $M(t) = e^{At} = \left[ {\begin{array}{*{20}c}   
     {m_{11}(t)} & {m_{12}(t)} \\
     {m_{21}(t)} & {m_{22}(t)} \\
   \end{array}
}\right].$  From $Ae^{At} = e^{At}A$ we first find 
\begin{equation}\label{E:m}
m_{22}(t) = m_{11}(t) ,\,\,m_{12}(t) =\frac{c}{b}m_{21}(t)
\end{equation}
Then 
\begin{equation}  \label{E:M'}
M'(t) = Ae^{At} = AM(t)=\left[ {\begin{array}{*{20}c}   
	{am_{11}(t)+cm_{21}(t)} & {\frac{ac}{b}m_{21}(t)+cm_{11}} \\
	{bm_{11}(t)+am_{21}(t)} & {cm_{21}(t)+am_{11}(t)} \\
	\end{array}
}\right].
\end{equation}
This yields a pair of linear first order DE's which can be solved for $m_{11}(t)\,,\,m_{21}(t)$. The results are as follows:
\begin{itemize}
\item For $c>0$,  $M(t) = \left[ {\begin{array}{*{20}c}   
	{e^{at}\cosh(b \lambda t)} & {\lambda e^{at}\sinh(b \lambda t)} \\
	{\frac{1}{\lambda} e^{at}\sinh(b \lambda t)} & {e^{at}\cosh(b \lambda t)} \\
	\end{array}
}\right]$ where $\lambda = \sqrt{c/b}$.
\item For $c<0$,  $M(t) = \left[ {\begin{array}{*{20}c}   
	{e^{at}\cos(b \lambda t)} & {-\lambda e^{at}\sin(b \lambda t)} \\
	{\frac{1}{\lambda} e^{at}\sin(b \lambda t)} & {e^{at}\cos(b \lambda t)} \\
	\end{array}
}\right]$ where $\lambda = \sqrt{-c/b}$.
\item For $c=0$,  $M(t) = \left[ {\begin{array}{*{20}c}   
	{e^{at}} & {0} \\
	{bt e^{at}} & {e^{at}} \\
	\end{array}
}\right].$ 
\item Notice that in all cases $0 < \lambda \leq 1$ and $\text{det}(M(t)) = e^{2at}$.
\item Replacing $t$ by $-t$, notice that
\begin{gather} \label{E:m-}
m_{11}(-t)=m_{22}(-t)= e^{-2at}m_{11}(t) \notag \\
m_{21}(-t)=-e^{-2at}m_{21}(t)\\
m_{12}(-t)= -\frac{c}{b}e^{-2at}m_{21}(t). \notag
\end{gather}
\end{itemize}

The $(x_1,x_2,x_3)$ are global coordinates for $G\left( A \right)$ and $\left( {\partial /\partial x_1,\partial /\partial x_2,\partial /\partial x_3}\right)$ is a global frame for the tangent bundle.  Identifying the Lie algebra $\mathfrak{g}(A)$ with the tangent space at the identity $G_A (0,0,0)$ gives $e_i =\partial /\partial x_i|_{(0,0,0)} $.

If we left translate $\left( {e_1 ,e_2 ,e_3 } \right) = \left( {\partial /\partial x_1,\partial /\partial x_2,\partial /\partial x_3} \right)|_{(0,0,0)} $ we get a global, left-invariant, orthonormal frame $\left( {L_1 ,L_2 ,L_3 } \right)$ for the tangent bundle.  To express $\left( {L_1 ,L_2 ,L_3 } \right)$ in terms of the frame $\left( {\partial /\partial x_1,\partial /\partial x_2,\partial /\partial x_3} \right)$, we have 
\[
\left[ {\begin{array}{*{20}c} 
	{L_1 }  \\
	{L_2 }  \\
	{L_3 }  \\
	
	\end{array} } \right]_{(x_1 ,x_2 ,x_3 )}  = \psi _{g} ^* \left[ {\begin{array}{*{20}c}
	{\partial /\partial x_1}  \\
	{\partial /\partial x_2}  \\
	{\partial /\partial x_3}  \\
	
	\end{array} } \right]_{(0,0,0)} 
\]
where $\psi _{g } :G(A) \to G(A)$ is left translation by $g  = G_A (x_1 ,x_2 ,x_3 )$.

Evaluating gives
\begin{equation}
\left[ {\begin{array}{*{20}c}
	{L_1 }  \\
	{L_2 }  \\
	{L_3 }  \\
		\end{array} } \right]_{(x_1 ,x_2 ,x_3 )}
 = \left[ {\begin{array}{*{20}c}
	{\partial /\partial x_1}  \\
	{M(x_1 )^T \left[ {\begin{array}{*{20}c}
			{\partial /\partial x_2}  \\
			{\partial /\partial x_3}  \\
						\end{array} } \right]}  \\
	\end{array} } \right]  \notag
\end{equation}
 
or explicitly
\begin{align} \label{E:Left}
L_1 &= \partial /\partial x_1  \notag\\
L_2 &= m_{11}(x_1) \partial / \partial x_2 + m_{21}(x_1) \partial / \partial x_3   \\
L_3 &= m_{12}(x_1) \partial / \partial x_2 + m_{22}(x_1) \partial / \partial x_3 \notag   . 
\end{align} 
Solving then yields 
\begin{equation}  \label{E:partial}
\left[ {\begin{array}{*{20}c}
	{\partial /\partial x_1}  \\
	{\partial /\partial x_2}  \\
	{\partial /\partial x_3}  \\
	\end{array} } \right]
       = \left[ {\begin{array}{*{20}c}
	{L_1 }  \\
	{M( - x_1 )^T \left[ {\begin{array}{*{20}c}
			{L_2 }  \\
			{L_3 }  \\
			
			\end{array} } \right]}  \\
	
	\end{array} } \right]
\end{equation}
or explicitly, using (\ref{E:m-}),
\begin{align} \label{E:part}
 \partial /\partial x_1 &= L_1 \notag\\
\partial /\partial x_2  &=  e^{-2ax_1}m_{11}(x_1)L_2 -e^{-2ax_1}m_{21}(x_1)L_3  \\
\partial/\partial x_3 &= -e^{-2ax_1}\frac{c}{b}m_{21}(x_1)L_2 +e^{-2ax_1}m_{11} (x_1)L_3 \notag   . 
\end{align} 

If we right translate $(e_1,e_2,e_3)$ we get a global right-invariant frame $(R_1,R_2,R_3)$ for the tangent bundle.  These right-invariant vector fields $R_i$ turn out to be Killing fields for the left-invariant metric on $G$.  To express $\left( {R_1 ,R_2 ,R_3 } \right)$ in terms of the frame $\left( {\partial /\partial x_1,\partial /\partial x_2,\partial /\partial x_3} \right)$, we have 
\[
\left[ {\begin{array}{*{20}c} 
	{R_1 }  \\
	{R_2 }  \\
	{R_3 }  \\
	
	\end{array} } \right]_{(x_1 ,x_2 ,x_3 )}  = \phi _{g} ^* \left[ {\begin{array}{*{20}c}
	{\partial /\partial x_1}  \\
	{\partial /\partial x_2}  \\
	{\partial /\partial x_3}  \\
	
	\end{array} } \right]_{(0,0,0)} 
\]
where $\phi _{g } :G(A) \to G(A)$ is right translation by $g  = G_A (x_1 ,x_2 ,x_3 )$.
A lengthy evaluation gives
\[\begin{gathered}
\left[ {\begin{array}{*{20}c}
	{R_1 }  \\
	{R_2 }  \\
	{R_3 }  \\
	
	\end{array} } \right]_{(x_1 ,x_2 ,x_3 )}
= \left[ {\begin{array}{*{20}c} 
	{\partial /\partial x_1 + (a x_2 + c x_3) \partial / \partial x_2 + (b x_2 + a x_3) \partial / \partial x_3}   \\
{\partial / \partial x_2}  \\
	{\partial / \partial x_3}  \\
	
	\end{array} } \right] 
\end{gathered} 
\]

Using (\ref{E:part}), we can then express the frame $\{R_i\}$ in terms of the left-invariant frame $\{L_i\}$:

\begin{multline} 
R_1=L_1+e^{-2ax_1}\left[m_{11}(x_1)(ax_2+cx_3)-\frac{c}{b}m_{21}(x_1)(bx_2+ax_3)     \right]L_2  \\
+e^{-2ax_1}\left[-m_{21}(x_1)(ax_2+cx_3)+  m_{11}(x_1)(bx_2+ax_3) \right]L_3 \notag
\end{multline} 
\begin{align} \label{E:Right}
R_2 &= e^{-2ax_1}m_{11}(x_1)L_2-e^{-2ax_1}m_{21}(x_1)L_3  \\
R_3 &=-e^{2ax_1}\frac{c}{b} m_{21}(x_1)L_2 + e^{-2ax_1}m_{11}(x_1)L_3. \notag
\end{align}

\section{Isometry groups and Killing vector fields}

The isometries of a Lie group $G$ with a left-invariant metric $g$ themselves form a Lie group, $\text{Isom}(G,g)$.  Given a one-parameter subgroup $\{\psi_t:G\rightarrow G:t \in \mathbb{R}\}$ of $\text{Isom}(G,g)$, we get a Killing vector field $X = X_i \partial / \partial x_i$ on $G$, where $X_i = d(x_i\circ \psi_t)/dt\arrowvert_{t=0}$. 

If $\psi_t$ is the product of two one parameter subgroups, $\psi_t = \psi_{1,t} \circ \psi_{2,t}$, then the Killing vector field $X$ corresponding to $\psi_t$ is given by $X = X_1 + X_2$, where $X_i$ is the Killing vector field correponding to $\psi_{i,t}$.

For any $x \in G$, left translation by $x$ gives an isometry $L_x$ of $G$, and $G$ can be regarded as a subgroup of $\text{Isom}(G,g)$. Then a one parameter subgroup of $G$ will give a one parameter subgroup of $\text{Isom}(G,g)$ and hence a Killing vector field.  The one parameter subgoups of $G$ correspond to elements of the Lie algebra $\mathfrak{g}$.  The Killing vector field corresponding to an element $e \in \mathfrak{g}$ is just the right invariant vector field obtained by right translating $e$. In particular, the basis elements $e_1,e_2,e_3$ give rise to the Killing fields $R_1,R_2,R_3$.

Write $\text{Isom}_e(G,g)$ for the subgroup of $\text{Isom}(G,g)$ consisting of isometries leaving the identity $e \in G$ fixed.
Then $\text{Isom}(G,g)$ is a product $ \text{Isom}_e(G,g) \cdot G$, and any Killing vector field can be written as a sum of a Killing vector field vanishing at $e$ and a right-invariant vector field. 

$X$ is a Killing field for $g$ if the Lie derivative of $g$ with respect to $X$ vanishes, that is if $(\mathcal{L}_Xg)(Y,Z) = 0$ for all $Y,Z$.  We express this condition in terms of the orthonormal frame $L_1,L_2,L_3$ of left-invariant fields.
\begin{proposition} \label{P:Killing}
A vector field $X=X^iL_i$ is a Killling vector field if and only if 
\begin{equation} \label{E:Kill}
L_j(X^k) + L_k(X^j) + X^i(c^j_{ki} + c^k_{ji}) = 0
\end{equation}
for all $j,k$.
\end{proposition}	
\begin{proof}
We have
\[\mathcal{L}_X(g(Y,Z)) =(\mathcal{L}_Xg)(Y,Z)+g(\mathcal{L}_X(Y),Z) +g(Y,\mathcal{L}_X(Z)) .\]
For the standard torsion free, inner product preserving connection on $G$ we have
\[(\mathcal{L}_X)(g(Y,Z))=X(g(Y,Z))=g(\nabla_XY,Z)+g(Y,\nabla_XZ)\]
and
\[\mathcal{L}_X(Y)=[X,Y]=\nabla_XY-\nabla_YX.\]
Then the condition $(\mathcal{L}_Xg)(Y,Z) = 0$ reduces to \[g(\nabla_YX,Z) + g(Y,\nabla_ZX) = 0.\]
Take $Y = L_j,Z= L_k$, write $X = X^i L_i$, and recall that the frame $\{L_i\}$ is orthonormal.  Then
\[g(\nabla_{L_j}X,L_k) =g(L_j(X^i)L_i+\Gamma^l_{ji}L_l,L_k)= L_j(X^k) + X^i \Gamma^k_{ji}\]
and similarly
\[g(L_j,\nabla_{L_k}X) = L_k(X^j) + X^i \Gamma^j_{ki}.\]
Here the symbols $\Gamma^j_{ik}$ are defined by $ \nabla_{L_k}L_i=\Gamma^j_{ki}L_j  $ and for this orthnormal frame are given by $\Gamma^j_{ki} = \frac{1}{2}(c^j_{ki}-c^k_{ij} + c^i_{jk})$. (Recall that the Lie algebra of left-invariant vector fields is isomorphic to $\mathfrak{g}$, so $[L_i,L_j] = c^k_{ij}L_k $.)  Notice that $\Gamma^k_{ji} + \Gamma^j_{ki} = c^j_{ki} + c^k_{ji}$.  Then the condition for a Killing field becomes
\[L_j(X^k) + L_k(X^j) + X^i(c^j_{ki} + c^k_{ji}) = 0\]
for all $j,k$ as claimed.
\end{proof}	
If we define matrices $D(X)$ and $C(X)$ for any vector field $X=X^iL_i$ by
\begin{equation}  \label{E:DC}
D(X)_{jk} = L_jX^k\,,\,\,\,\,C(X)_{jk}= X^i(c^j_{ki} + c^k_{ji})
\end{equation}
then the condition for $X$ to be a Killing field is just
\begin{equation}  \label{E:K}
D(X) + D(X)^T + C(X) = 0.
\end{equation}

Recall that the Lie algebra structure constants (for $\mathfrak{g}$ or for the left-invariant vector fields on $G$ ) are given by 
\begin{align}
c^2_{12} &= -c^2_{21} =a  &  c^2_{13} &= -c^2_{31} =c \notag \\
c^3_{12} &= -c^3_{21} = b &  c^3_{13} &= -c^3_{31} =a. 
  \notag
\end{align}
For future use we compute the matrix
\begin{equation} \label{E:Xc}
C(X) = \left[ {\begin{array}{*{20}c}
	{0} & {aX^2+cX^3} & {bX^2+aX^3}	\\
	{aX^2+cX^3} & {-2aX^1} & {-(b+c)X^1} \\
	{bX^2+aX^3} & {-(b+c)X^1} & {-2aX^1}	
	\end{array} } \right]
\end{equation}

As a tedious exercise, one can check that each right invariant field $R_i$ satisfies the condition (\ref{E:Kill}).

A Killing vector field $X$ is ``parallel at the identity'' if $\nabla_Y (X)$ vanishes at $e$ for every vector field $Y$.  Evidently it is enough to check that $\nabla_{L_i} (X) \arrowvert _e = 0$ for each $L_i$.  Writing $X = X^j L_j$, we find
\[\nabla_{L_i} (X) = L_i(X^j) L_j + X^j \nabla_{L_i} (L_j) = L_i(X^k) L_k + X^j \Gamma^k_{ij} L_k.\]
Define a matrix $P(X)$ by
\begin{equation} \label{E:X(k,i)}
P(X)_{ki} = \{ L_i(X^k) + X^j \Gamma^k_{ij}\} \arrowvert _e
\end{equation}
(where $\arrowvert_e$ means evaluated at the identity $e=(0,0,0) $).  We then have
\begin{proposition} \label{P:par}
A Killing vector field $X=X^iL_i$ is parallel at the identity if and only if $P(X) = 0$.
\end{proposition}
Notice that $P(X)$ is linear in $X$, that is, if $X = \alpha_r X_r$, then $P(X) = \alpha_r P(X_r)$.

If we define a matrix $\Gamma(X)$ by $\Gamma(X)_{ki} =X^j \Gamma^k_{ij}\ \arrowvert _e $, then we have
\[P(X) = D(X)^T + \Gamma(X).\]
Using  $\Gamma^j_{ki} = \frac{1}{2}(c^j_{ki}-c^k_{ij} + c^i_{jk})$ we compute for future use the matrix 
\begin{equation} \label{E:Xgamma}
\Gamma(X) =	 \left[ {\begin{array}{*{20}c}
			{0} & {\left( aX^2+\frac{1}{2}(b+c) X^3\right)\arrowvert_e} & {\left( \frac{1}{2}(b+c)X^2+aX^3\right)\arrowvert_e}	\\
			{-\frac{1}{2}(b-c)X^3\arrowvert_e} & {-aX^1\arrowvert_e} & {-\frac{1}{2}(b+c)X^1\arrowvert_e} \\
			{\frac{1}{2}(b-c)X^2\arrowvert_e} & {-\frac{1}{2}(b+c)X^1\arrowvert_e} & {-aX^1\arrowvert_e}	
	\end{array} } \right]
\end{equation}
Using the explicit values found above in (\ref{E:Right}) for the right invariant vector fields $R_j$, we can calculate (see section 12) the $P(R_j)$ matrices for the $R_j$ (which span the space of all right invariant vector fields).  The results are as follows:
\begin{gather} 
 P(R_1) =  \left[ {\begin{array}{*{20}c}
	{0} & {0} & {0}	\\
	{0} & {0} & {-\frac{b-c}{2}} \\
	{0} & {\frac{b-c}{2}} & {0}	
		\end{array} } \right]   \notag \\  
 P(R_2) =  \left[ {\begin{array}{*{20}c}
	{0} & {a} & {\frac{b+c}{2}}	\\
	{-a} & {0} & {0} \\  
	{-\frac{b+c}{2}} & {0} & {0}	
	\end{array} } \right]    \label{R(k,i)}   \\  
 P(R_3) =  \left[ {\begin{array}{*{20}c}
	{0} & {\frac{b+c}{2}} & {a}	\\
	{-\frac{b+c}{2}} & {0} & {0} \\
	{-a} & {0} & {0}	
	\end{array} } \right]  \notag  
\end{gather}

For a Killing vector field $X$ vanishing at $e$ which  corresponds to a one parameter subgroup $\psi_t$ of $\text{Isom}_e(G,g)$, we can compute $P(X)$ as follows. Since each $ x_j \circ \psi_t$ vanishes at $e = (0,0,0)$, we can write
\[x_j \circ \psi_t = a(t)_{ji}x_i + r_t^j   \]
where each $r_t^j$ and all derivatives $\partial(r_t^j)/\partial x_i$ vanish at $e$.  The matrix $a(t)$ is just the Jacobean of $\psi_t$ evaluated at $e$, that is,
\begin{equation}  \label{E:a(t)}
a(t)_{ji} = \partial(x_j \circ \psi_t)/\partial(x_i) \arrowvert_e.
\end{equation}
Each matrix $a(t)$ gives a linear map of the tangent space at $e$ to itself, and, since each $\psi_t$ is an isometry, $a(t)$ will be an orthogonal transformation on $\mathfrak{g}$ taking the standard basis $e_1,e_2,e_3$ to a new orthonormal basis $b(t)_1,b(t)_2,b(t)_3:$
\begin{equation}  \label{E:basischange}
b(t)_i = (\psi_t)_*(e_i) = a(t)_{ji} e_j.
\end{equation}
This one parameter subgroup $a(t)$ of the orthogonal group $O(3)$ determines $P(X)$ as follows. Writing $X = F^j \partial/\partial x_j$ where $F^j = (x_j \circ \psi_t)' \arrowvert _{t=0}\,$, we find
\[F^j = [(a(t)_{ji})'x_i + (r_t^j)']\arrowvert _{t=0} =\bar{a}_{ji} x_i +\bar{r}^j  \]
where $\bar{a} = (a(t))'\arrowvert _{t=0}$ and each $\bar{r}^j$ and all its first order partials vanish at $e$.  Then, using (\ref{E:partial}), $X = F^j \partial/\partial x_j = X^j L_j$ where
\[X^1 = F^1 = \bar{a}_{1i}x_i + \bar{r}^j\] and \[\left[\begin{array}{*{20}c}
X^2 \\
X^3
\end{array} \right] =
M(-x_1) \left[\begin{array}{*{20}c}
{F^2}\\
{F^3}
\end{array} \right] =
M(-x_1) \left[\begin{array}{*{20}c}
{\bar{a}_{2i}x_i + \bar{r}^2}\\
{\bar{a}_{3i}x_i + \bar{r}^3}
\end{array} \right] .\]

Recalling that each $X^i$ vanishes at $e$ and that $L_i = \partial/\partial x_i$ at $e$, the formula for $P(X)_{ki}$ becomes 
\begin{equation} P(X)_{ki} = \{ L_i(X^k) + X^j \Gamma^k_{ij}\} \arrowvert _e  = \{\partial(X^k)/\partial x_i\} \arrowvert _e.
\end{equation} 
But then since $M(0) = I$ and all $\bar{r}^j$ and their first order partial derivatives vanish at $e = (0,0,0)$, we have the following result:
\begin{proposition} \label{P:Xvanate}
Let $\psi_t$ be a one parameter subgroup of $\text{Isom}_e(G,g)$ with corresponding Killing vector field $X=X^iL_i$ vanishing at $e$.  Then 	   
\begin{equation} \label {E:Xvanate}
P(X)_{ki} = \bar{a}_{ki} = (a(t)_{ki})'\lvert_{t=0}  = \frac{d}{dt} \left[\partial(x_j \circ \psi_t)/\partial(x_i)\arrowvert_e \right]  \lvert _{t=0}.
\end{equation}    
\end{proposition}
As a corollary of propositions (\ref{P:par}) and (\ref{P:Xvanate}) we have
\begin{corollary}  \label{C:Xvanate}
If $a(t)$ is independent of $t$, then $P(X) = 0 $ and $X$ is parallel at the identity.
\end{corollary}

\section{The index of symmetry $i(G)$ and subindex $\bar{i}(G)$.}

Write $S_e$ for the subspace of $\mathfrak{g}$ (considered as the tangent space to $G$ at $e$) consisting of all tangent vectors $x=X\arrowvert_e$ where $X$ is a Killing vector field parallel at $e$.  Then the index of symmetry for $G$ is $i(G) =i(G,g)= \text{dim}(S_e)$.

Let $\bar{S}_e$ be the subspace of $S_e$ consisting of all vectors $x = R\arrowvert_e$ where $R$ is a right invariant vector field parallel at $e$. Then define the \textit{subindex} $\bar{i}(G) = \text{dim}(\bar{S}_e)$.  Given the explicit formulas (\ref{R(k,i)}) for the $R_i$, which span the right invariant vector fields, we can compute $\bar{i}(G)$ with relative ease. The following proposition suggests the usefulness of $\bar{i}(G)$.
\begin{proposition} \label{P:5.1}
If every Killing field which vanishes at $e$ is actually parallel at $e$, then $\bar{i}(G) = i(G)$.
\end{proposition}	
\begin{proof}
	Take $x = Y\arrowvert_e \in S_e$ where $Y$ is a Killing vector field parallel at $e$.  We must show that $x \in \bar{S}_e$.  Write $Y = R + X$ where $R$ is a right invariant field and $X$ is a Killing field vanishing at $e$. Then $x = Y\arrowvert_e = R\arrowvert_e$, so if $R$ is parallel at $e$, then $x \in \bar{i}(G)$ and we are done.  But $Y$ is parallel at $e$ and $X$ is parallel at  $e$ by hypothesis, so $R = Y - X$ is in fact parallel at $e$.	
\end{proof}
Then using corollary \ref{C:Xvanate} we get
\begin{corollary} \label{C:5.1}
If for every one parameter subgroup $\psi_t$ of $\text{Isom}_e(G,g)$ the matrices $a(t)$ are independent of $t$, then $\bar{i}(G) = i(G)$.
\end{corollary}
To calculate $\bar{i}(G)$, note that a Killing field $X = \alpha_j R_j$ is parallel at $e$ if and only if $\alpha_j P(R_j) = 0$.  $X$ then contributes $\alpha_j e_j$ to $\bar{S}_e$.   Using the formulas (\ref{R(k,i)}) for $P(R_j)$ we get three independent equations for the $\alpha_j$:
\begin{gather}
\alpha_1 \left (\dfrac{b-c}{2}\right) = 0  \notag \\
\alpha_2\, a + \alpha_3 \left(\dfrac{b+c}{2} \right) = 0   \\
\alpha_2 \left(\dfrac{b+c}{2}\right) + \alpha_3 \, a = 0  .  \notag
\end{gather}

So when $b=c$, we can take $\alpha_1 = 1$ and $\alpha_2 = \alpha_3 = 0$ and get the vector $e_1 \in \bar{S}_e$. This occurs for the algebras $\mathfrak{g}(K)$ with $K > 0$ for the triples $(a,b,b)$, for the algebra $\mathfrak{g}_I$ for the triples $(a,0,0)$, for the unimodular algebra $\mathfrak{e}(1,1)$ for the triples $(0,b,b)$, and for the abelian algebra with triple $(0,0,0)$.  

There is a nontrivial solution for $\alpha_2$ and $\alpha_3 $ only if the determinant of the coefficents, $a^2 - \left((b+c)/2 \right)^2 $, vanishes, that is, if $a = \left(\dfrac{b+c}{2} \right)$. In this case, when $2a = b+c$, we can take $\alpha_2 = -\alpha_3 = 1$ and get the vector $e_2 - e_3 \in \bar{S}_e$.  This occurs for the algebras $\mathfrak{g}(K)$  with $K \leq 1$ and the triples $(a,a+a\sqrt{1-K},a-a\sqrt{1-K})$ and for the unimodular algebra $\mathfrak{e}(2)$ with the triples $(0,b,-b)$ (and for the abelian algebra with the triple $(0,0,0)$).

The vectors found above give a basis for $\bar{S}_e$, so we obtain the following results for the subindex $\bar{i}(G)$:

\begin{itemize}
\item For $\mathfrak{g}(K)$ with $K>0\,,\,K\neq1$, triple $(a,b,b)$:  $\bar{S}_e = \text{span}(e_1)$ and $\bar{i}(G) = 1$.
\item For $\mathfrak{g}(1)$ with triple $(a,b,b)$:  $\bar{S}_e = \text{span}(e_1,e_2-e_3)$ and $\bar{i}(G) = 2$.
\item For $\mathfrak{g}(K)$ with $K < 1$, triple $(a,a+a\sqrt{1-K},a-a\sqrt{1-K})$:  $\bar{S}_e = \text{span}(e_2-e_3)$ and $\bar{i}(G) = 1$. 
\item For $\mathfrak{g}_I$ with triple $(a,0,0)$:  $\bar{S}_e = \text{span}(e_1)$ and $\bar{i}(G) = 1$. 
\item For the unimodular algebra $\mathfrak{e}(1,1)$ with triple $(0,b,b)$:  $\bar{S}_e = \text{span}(e_1)$ and $\bar{i}(G) = 1$.
\item For the unimodular algebra $\mathfrak{e}(2)$ with triple $(0,b,-b)$:  $\bar{S}_e = \text{span}(e_2-e_3)$ and $\bar{i}(G) = 1$.
\item For the abelian algebra: $\bar{S}_e = \mathfrak{g}$ and $\bar{i}(G) = 3$.
\item For all other cases we have $\bar{S}_e = \{0\}$ and $\bar{i}(G) = 0$. 
\end{itemize}

\section{Killing fields vanishing at $e$}

In this section we show how the Ricci tensor, $\Ric$, of $G$ evaluated at $e = (0,0,0) $ can influence the contribution of a Killing vector field vanishing at $e$ to the index of symmetry $i(G)$.

First note that for any one parameter group of isometries $\psi_t:G\rightarrow G$ taking $e$ to $e$, each $\psi_t$ will preserve the Ricci tensor, $\psi_t^*\Ric = \Ric$, and, in particular, $\psi_t^*\Ric(e) = \Ric(e)$.  
Define matrices for $\Ric(e)$ and $\psi_t^*\Ric(e)$ for the standard basis $\mathbf{e}= (e_1,e_2,e_3)$:  $R_{ij} = \Ric(e)(e_i,e_j)$ and $(\psi_t^*R)_{ij} = \psi_t^*\Ric(e)(e_i,e_j)$.  (So we must have $\psi_t^*R  = R$.)  Then since $\psi_t(e)=e$, we find using equation (\ref{E:basischange})
\begin{align}
(\psi_t^*R)_{ij} &= \psi_t^*\Ric(e)(e_i,e_j) = \Ric(e)((\psi_t)_*e_i,(\psi_t)_*e_j )   \notag \\
       &= \Ric(e)(a(t)_{ki}e_k,a(t)_{lj}e_l) \notag \\
       &= a(t)_{ki}a(t)_{lj}\Ric(e)(e_k,e_l) \notag \\ 
       &=a(t)_{ki}a(t)_{lj}R_{kl} \notag
\end{align}  
where $a(t)_{ki} = \partial (x_k\circ \psi_t ) /\partial x_i$. In matrix terms we have 
\[R =\psi_t^*R = a(t)^T R\,\, a(t).\]
Recall that the matrices $a(t)$ are orthogonal, so $a(t)^T =a(t)^{-1}$.  Then
\[R = a(t)^{-1}R\,\,a(t)=a(t)R\,\,a(t)^{-1}.\]

For the standard basis $e = \{e_i\}$ for a metric Lie algebra $\mathfrak{g}$ corresponding to the triple $(a,b,c)$, the Ricci matrix can be calculated as 
\begin{equation} \label{Ricci}
R = \left[ {\begin{array}{*{20}c}
	{ - 2a^2  - \frac{{(b + c)^2 }}
		{2}} & 0 & 0  \\
	0 & { - 2a^2  - \frac{{b^2  - c^2 }}
		{2}} & { - a(b + c)}  \\
	0 & { - a(b + c)} & { - 2a^2  - \frac{{c^2  - b^2 }}
		{2}}  \\
	\end{array} } \right].
\end{equation} 

One can diagonalize this matrix by a ``rotation'' of the $e_2e_3$-plane in $\mathfrak{g}$ through an angle $\theta$ where $\tan(2\theta) = 2a/(b-c)$:  Define a rotation matrix 
$r(\theta) = \left[ \begin{array}{*{20}c}
1 & 0 & 0 \\
0 & \cos(\theta) & -\sin(\theta) \\
0 & \sin(\theta) & \cos(\theta)
\end{array} \right]$.  Notice that $r(\theta)^{-1} = r(-\theta) = r(\theta)^T$, so $r(\theta)$ is orthogonal.  Define a smooth change of coordinates with $y_j = r(\theta)_{i,j}x_i$.  Then the new orthonormal basis $d = \{d_j\}$ for $\mathfrak{g}$ where $d_j = \partial/\partial y_j \lvert_e$ is obtained by rotating the $e_2e_3$ plane through the angle $\theta$:  $d_j = r(\theta)_{i,j}e_i$.  Notice that $e_k = r(-\theta)_{j,k}d_j$.  Then when $\tan(2\theta) = 2a/(b-c)$ one finds that the Ricci matrix relative to this new basis $d$ has the diagonal form
\begin{equation} 
R_d=  \left[ {\begin{array}{*{20}c}
	R_{11} & 0 & 0  \\
	0 & R_{22} & 0  \\
	0 & 0 & R_{33} \\
	\end{array} } \right]. \notag
\end{equation} 
where
\begin{gather} \label{E:R_ii}
R_{11} = - 2a^2  - \frac{(b + c)^2 }{2} \notag \\
R_{22} = - 2a^2  - \frac{(b + c)}{2}\sqrt {4a^2  + (b - c)^2 }  \\
R_{33} = - 2a^2  + \frac{(b + c)}{2}\sqrt {4a^2  + (b - c)^2 }\,\,.  \notag
\end{gather}

The matrices $R$ and $R_d$ are related by 
\[R_d = r(\theta)^TR\,\,r(\theta) = r(\theta)^{-1}R\,\,r(\theta)\]
where $r(\theta)$ is the orthogonal rotation matrix relating the standard basis $e_i$ to the ``diagonalizing'' basis $d_i$.  Then a calculation shows that $R=a(t)R\,\, a(t)^{-1}$ implies $R_d = a(d,t)^{-1}R_d\,\,a(d,t)$ where $a(d,t) = r(\theta)^{-1}a(t)^{-1}\,r(\theta)$. 
Then using proposition (\ref{P:Xvanate}) we have
\begin{proposition} \label{P:a(d,t)}
Let $\psi_t$ be a one parameter subgroup of $\text{Isom}_e(G,g)$ with corresponding Killing vector field $X=X^iL_i$ vanishing at $e$.  Then there exists a smoothly varying family of orthogonal matrices $a(d,t)$ with $a(d,0)= I$, the identity, such that

(1)  Each $a(d,t)$ commutes with the diagonalized Ricci matrix $R_d$, that is, $R_d = a(d,t)^{-1}R_d\,\,a(d,t)$.

(2)  $a(t)= r(\theta) a(d,t)^{-1}\,r(\theta)^{-1}$.

(3)  $P(X) = (a(t)'\lvert_{t=0})  = \left( r(\theta) (a(d,t)^{-1})'\lvert_{t=0}\,\,r(\theta)^{-1} \right)$.
\end{proposition}

As we will see, the condition that $a(d,t)$ commutes with the diagonal matrix $R_d$ places restrictions on $a(d,t)$ and therefor on $a(t)$ and $P(X)$.  

\section{Index of symmetry for generic $G$}

Define a metric Lie algebra $\mathfrak{g}$ (and the corresponding simply connected Lie group $G$) to be \textit{generic} if the diagonal entries $R_{11},R_{22}, R_{33}$ for the diagonalized Ricci operator $R_d$ are all distinct. Our first main result is the following:
\begin{proposition} \label{P:generic}
For generic $G$, $S_e = \bar{S}_e$ and $i(G) = \bar{i}(G)$.	
\end{proposition} 
\begin{proof}
For a generic $G$, take any one parameter group $\psi_t$ of isometries fixing $e$ with corresponding Killing vector field $X_e$ vanishing at $e$.   As in proposition (\ref{P:a(d,t)}) we have a smoothly varying family $a(d,t)$ of orthogonal matrices which commute with the diagonal matrix $R_d$.  Since $G$ is generic, the diagonal entries of $R_d$ are all distinct.  Then one finds that commuting with $R_d$ requires each $a(d,t)$ to be diagonal.  But then since $a(d,t)$ is orthogonal, each diagonal entry must be $\pm1$.  Finally, since $a(d,0)$ is the identity and $a(d,t)$ varies continuously with $t$, all diagonal entries must be $1$, that is, $a(d,t) = I=a(d,t)^{-1}$ for all $t$.  But then $a(t)= r(\theta) a(d,t)^{-1}\,r(\theta)^{-1}= I$ is independent of $t$.  Then $\bar{i}(G) = i(G)$ by corollary (\ref{C:5.1}).
\end{proof}

Combined with the results of section 6, the proposition gives $i(G)$ for all but the nongeneric $G$.  We  now find the triples $(a,b,c)$ which give nongeneric $G$, that is, for which the Ricci matrix $R_d$ does not have distinct diagonal entries.

First notice that if $b+c = 0$, then $R_{11} = R_{22} = R_{33} = -2a^2$.

If $b+c \neq 0$, then $R_{11} = R_{22}$ if and only if $a^2 = bc$.  

If $b+c \neq 0$, then $R_{11} \neq R_{33}$.

If $b+c \neq 0$, then $R_{22} = R_{33}$ if and only if $a=0$ and $c=b$.

We can then catalog the nongeneric (and nonabelian) metric Lie algebras as follows:

\begin{itemize}
\item $\mathfrak{g}(1)$ for any triple $(a,b,c)$ (with $a > 0, c = a^2/b$). 
\item $\mathfrak{g}(K)$ for $K<0$ and any triple $(a,b,-b)$ with $a>0, b = a \sqrt{-K}$.
\item $\mathfrak{g}_I$ for any triple $(a,0,0)$ (with $a>0$).
\item The unimodular algebra $\mathfrak{e}$(2) with any triple $(0,b,-b)$ ($b>0$).
\item The unimodular algebra $\mathfrak{e}(1,1)$ with any triple $(0,b,b)$ ($b>0$).
\item The Heisenberg algebra with any triple $(0,b,0)$ ($b>0$). 
\end{itemize}
We will compute $i(G)$ for these non-generic $\mathfrak{g}$ case by case.

\section{The case $\mathfrak{g}(1)$.}

Consider the metric Lie algebras $\mathfrak{g}(1)$ corresponding to a matrix
$A = \left[ 
   {\begin{array}{*{20}c}
   	   a & c  \\
   	   b & a  
	\end{array} }
       \right]$ 
where $a > 0$ and $a^2 = bc$.  In this section we will verify the following result:

\begin{proposition}
	Let $G$ be a simply-connected Lie group with left-invariant metric and let the corresponding metric Lie algebra be $\mathfrak{g}(1)$ with triple $(a,b,c)$, $a>0$ and $a^2=bc$.  Then:
	(1)  If $b=c$, $i(G)=3$.
	(2)  If $b\neq c$, $i(G) = 1$. 
\end{proposition}

\begin{proof}
In section 12 we sketch the proof that $X = X^iL_i$ is a Killing vector field when 
\[X^1 = bx_2 + ax_3\]
\[X^2 = \frac{1}{8b}\left[ e^{-2ax_1}\left\lbrace 4a(bx_2 + ax_3)^2 + (a^2 + b^2)/a \right\rbrace  +e^{2ax_1}(a^2 - 3b^2)/a + 2(b^2 - a^2)/a  \right] \]
\[X^3 = \frac{1}{8a}\left[ e^{-2ax_1}\left\lbrace 4a(bx_2 + ax_3)^2 + (a^2 + b^2)/a \right\rbrace  +e^{2ax_1}(b^2 - 3a^2)/a + 2(a^2 - b^2)/a      \right]\, .\]

One easily checks that this vector field $X$ vanishes at $e=(0,0,0)$.  Then one finds (see section 12)
\[P(X) = D(X)^T\mid_e =
\left[ \begin{array}{*{20}c}
0 & b & a \\
-b & 0 & 0 \\
-a & 0 & 0
\end{array} \right] 
. \]

Using the values found above in (\ref{R(k,i)}) for the right invariant fields $R_j$, we find for a linear combination $Y = \alpha_j R_j + \beta X$:
\[P(Y) = \alpha_j P(R_j) + \beta P(X) =\]
\[ \left[ \begin{array}{*{20}c}
0  & {\alpha_2 a + \alpha_3 \left( \dfrac{a^2 + b^2}{2b}\right)  + \beta b}  &  {\alpha_2 \left(\dfrac{a^2 + b^2}{2b}\right) + \alpha_3 a + \beta a} \\
{-\alpha_2 a - \alpha_3 \left(\dfrac{a^2 + b^2}{2b}\right)  -\beta b } & 0 & {-\alpha_1\left(\dfrac{b^2 - a^2}{2b}\right)} \\
{-\alpha_2 \left(\dfrac{a^2 + b^2}{2b}\right) - \alpha_3 a -\beta a} & \alpha_1 \left(\dfrac{b^2 - a^2}  {2b}\right) & 0 \\
\end{array} \right]  .\]
Then for $Y$ to be parallel at $e$ (that is, for $P(Y)=0$) we need
\[\alpha_1 \left( \dfrac{b^2-a^2}{2b} \right ) = 0\]
\[\alpha_2 a + \alpha_3 \left( \dfrac{a^2 + b^2}{2b}   \right) = -\beta b \]
\[\alpha_2 \left( \dfrac{a^2+b^2}{2b}\right) + \alpha_3 a  = -\beta a   \,\,. \]

When $a = b$ there is a solution with $\alpha_1 = 1$ and all other coefficients $0$.  Then $Y = R_1$ and $Y\lvert_e = e_1 \in S_e$.  But also if $a = b$ the last two equations both become $\alpha_2 b + \alpha_3 b = -\beta b $.  This gives two independent solutions:  $\beta = 0,\,\,\alpha_2 = 1,\,\,\alpha_3 = -1 $, for which $Y = R_2 - R_3$ and $Y\lvert_e = e_2 -e_3 \in S_e$ and $\beta = 1,\,\,\alpha_2 = 0,\,\,\alpha_3=-1$ for which $Y = -R_3 + X$ and $Y\arrowvert_e = -e_3 \in S_e$.  Then evidently $S_e$ is all of $\mathfrak{g}$ and $i(G) = 3$.

Now consider the case $a \neq b$.  We must then have $\alpha_1 = 0$.  The two by two coefficient determinant for the last pair of equations is nonzero, so there will be a unique solution for $\alpha_2, \alpha_3$ corresponding to any choice of $\beta$.  (Notice that for $\beta = 0$ the solution is $\alpha_2 = \alpha_3 = 0$, so there is no solution where $Y$ is right invariant.  That is, $\bar{S}_e = 0$, $\bar{i}(G)= 0$.)  However, for any nonzero $\beta$, say $\beta = 1$, there will be a unique solution for $\alpha_2, \alpha_3$. (In fact, for $\beta = 1$, the solution is $\alpha_2 = \left[-ab +a\left(\frac{a^2+b^2}{2b}   \right) \right]/D$ and $\alpha_3 = \left[-a^2 +b\left(\frac{a^2+b^2}{2b}   \right) \right]/D $ where $D$ is the determinant $D = a^2 -\left(\dfrac{a^2+b^2}{2b}  \right)^2$ .)  So there is a solution (unique up to nonzero scalar multiple) $Y = \alpha_2 R_2 + \alpha_3 R_3 + X$ with $Y\lvert_e = \alpha_2 e_2 + \alpha_3 e_3 \in S_e$.

We claim that up to scalar multiple this $Y\arrowvert_e$ is the only element of $S_e$, so that $i(G) = 1$ for this case.  In fact, take any Killing field $\bar{X}$ vanishing at $e$ and consider the corresponding one parameter family of isometries $\psi_t$.  For the appropriate coordinates, the Ricci operator is diagonalized as 
\begin{equation} 
R_d=  \left[ {\begin{array}{*{20}c}
	R_{11} & 0 & 0  \\
	0 & R_{22} & 0  \\
	0 & 0 & R_{33} \\
	\end{array} } \right].\notag
\end{equation} 
where
\[R_{11} = R_{22} = - 4a^2  - \frac{(b - c)^2 }{2}\]
\[R_{33} = \frac{(b - c)^2}{2}. \]
As seen in section 8, the group $\psi_t$ gives rise to a one parameter group of orthogonal matrices $a(d,t)$ which commute with $R_d$.  To commute with $R_d$, evidently each orthogonal $a(d,t)$ must leave $d_3$ fixed and perform a rotation in the $d_1-d_2$ plane.  That is,
$a(d,t) = \left[\begin{array}{*{20}c}
\cos(\lambda t) & -\sin(\lambda t) & 0 \\
\sin(\lambda t)  & \cos(\lambda t) & 0 \\
0  &  0  &  1
\end{array}\right]$.
Then $a(d,t)^{-1} = \left[\begin{array}{*{20}c}
\cos(\lambda t) & \sin(\lambda t) & 0 \\
-\sin(\lambda t)  & \cos(\lambda t) & 0 \\
0  &  0  &  1
\end{array}\right]$
and $\left( a(d,t)^{-1} \right)'\mid_{t=0} = \left[\begin{array}{*{20}c}
0 & \lambda  & 0 \\
-\lambda  & 0 & 0 \\
0  &  0  &  0
\end{array}\right].$
Then by proposition (\ref{P:a(d,t)}), $P(\bar{X})$ is completely determined (up to a scalar factor $\lambda$) independent of the choice of the field $\bar{X}$. In particular, $P(\bar{X})$ must equal $P(X)$ up to a scalar factor. Now take any $x = \bar{Y}\arrowvert_e \in S_e$ where $\bar{Y} = R + \bar{X}$ is a Killing field parallel at $e$ with $R$ right invariant and $\bar{X}$ vanishing at $e$. Then $0 =P(\bar{Y}) = P(R) + P(\bar{X}) = P(R)  + \kappa P(X)$
for some scalar $\kappa$.  Then $Z = R + \kappa X$ is parallel at $e$, so $Z$ must be a scalar multiple of the field $Y$ found above, say $Z=\gamma Y$. Then $Z\arrowvert_e = R\arrowvert_e = \bar{Y}\arrowvert_e =x$, but also $Z\arrowvert_e = \gamma Y\arrowvert_e = \gamma (\alpha_2e_2+\alpha_3e_3)$. So every element $x \in S_e$ is a scalar multiple of $\alpha_2 e_2 + \alpha_3 e_3$.  Then $i(G) = 1$ for the cases $a \neq b$ as claimed. 
\end{proof}

\section{The cases $a>0$, $b+c = 0$.}

In this section we consider the nongeneric metric Lie algebras $\mathfrak{g}(K)$, $K<1$, with $b+c = 0$ and the special case $\mathfrak{g}_I$.  We will derive the following result:
\begin{proposition}
	Let $G$ be a simply-connected Lie group with left-invariant metric and let the corresponding metric Lie algebra be either $\mathfrak{g}(K)$ with $K<1$ and $b+c = 0$ or $\mathfrak{g}_I$.  Then $i(G)=3$.
\end{proposition}

The Ricci matrix for this case is already diagonalized for the standard coordinates with $d_i = e_i$:
 \begin{equation} 
 R = R_d=  \left[ {\begin{array}{*{20}c}
 	R_{11} & 0 & 0  \\
 	0 & R_{22} & 0  \\
 	0 & 0 & R_{33} \\
 	\end{array} } \right].  \notag
 \end{equation} 
 where
 \[R_{11} = R_{22} = R_{33}= - 2a^2 .\]
 All sectional curvatures are negative and equal to $-a^2$, so (by the Killing-Hopf theorem, see e.g. Wolf \cite{Wolf}) $G$ is isometric to a hyperbolic three-space $H^3 = H^3(a^2)$ (with a metric having constant negative curvature $-a^2$).  So it is not surprising that $G$ has a maximal amount of symmetry, $i(G)=3$.  
 
 A standard Poincare model (see  e.g. Wielenberg \cite{Wiel}) identifies $H^3(a^2)$ with the subspace $\{(y_1,y_2,y_3): y_1 > 0\}$  of $\mathbb{R}^3$ with the metric $(dy_1^2 + dy_2^2 + dy_3^2)/(a^2y_1^2)$.  Meanwhile the metric for $G$ with the standard coordinates is $dx_1^2 + (dx_2^2 + dx_3^2)/e^{2ax_1}$.  Then one can check that the map $G\rightarrow H^3(a^2)$ given by 
\[y_1 = e^{ax_1}/a ,\,\, y_2 = x_2,\,\,y_3=x_3\]
is an isometry.  

It is known (see Wielenberg \cite{Wiel}) that the isometry group of $H^3$ (and therefor also of $G$) can be identified with the Lie group $PSL(2,\mathbb{C})$ of all two by two complex matrices with determinant 1.  So any one parameter subgroup of $PSL(2,\mathbb{C})$ should give a Killing vector field for $G$.

For the one parameter subgroup $\left[
\begin{array}{*{20}c}
1  &  0  \\
t  &  1 
\end{array}
    \right]$ one can show that the corresponding Killing field on $G$ is $X_1 = X_1^i L_i$ where 
\[ X_1^1 = \dfrac{-2x_2}{a} \]
\[ X_1^2 = e^{-ax_1}\left[\cos(bx_1)(x_3^2 - x_2^2)  -\sin(bx_1)(2x_2x_3) \right] + \dfrac{e^{ax_1}\cos(bx_1)}{a^2}   \]
\[X_1^3 = e^{-ax_1}\left[-\sin(bx_1)(x_3^2 - x_2^2) -\cos(bx_1)(2x_2x_3) \right] - \dfrac{e^{ax_1}\sin(bx_1)}{a^2}      \]
In section 12 we sketch a proof that $X_1$ is in fact a Killing field.

One can easily see that $X_1\arrowvert_e = e_2/a^2$.  In section 12 we compute 
\[P(X_1) = \left[ \begin{array}{*{20}c}
 0  &  -1/a   &  0 \\
 1/a  &  0  &  0  \\
 0  &  0  &  0
\end{array} \right] . \] 

Then let $Y_1 = X_1 + R_2/a^2$.  Using (\ref{R(k,i)}) and $b+c=0$, we find $P(Y_1) = 0$.  So $Y_1$ is parallel at $e$ and $Y_1|_e = 2e_2/a^2 \in S_e$.

Similarly, for the one parameter subgroup $\left[
\begin{array}{*{20}c}
1  &  0  \\
it  &  1 
\end{array}
\right]$ the corresponding Killing field on $G$ is $X_2 = X_2^i L_i$ where 
\[ X_2^1 = \dfrac{2x_3}{a} \]
\[ X_2^2 = e^{-ax_1}\left[\sin(bx_1)(x_3^2 - x_2^2)  +\cos(bx_1)(2x_2x_3) \right] - \dfrac{e^{ax_1}\sin(bx_1)}{a^2}   \]
\[X_2^3 = e^{-ax_1}\left[\cos(bx_1)(x_3^2 - x_2^2) -\sin(bx_1)(2x_2x_3) \right] - \dfrac{e^{ax_1}\cos(bx_1)}{a^2}      \]

One finds $X_2\arrowvert_e = -e_3/a^2$ and that 
\[P(X_2) = \left[ \begin{array}{*{20}c}
0  &  0   &  1/a \\
0  &  0  &  0  \\
-1/a  &  0  &  0
\end{array} \right] . \] 

Then if $Y_2 = X_2 - R_3/a^2$ we find $P(Y_2) = 0$.  So $Y_2$ is parallel at $e$ and $Y_2|_e =-2e_3/a^2 \in S_e$.

Finally, for the one parameter subgroup $\left[
\begin{array}{*{20}c}
e^{it}  &  0  \\
0  &  e^{-it} 
\end{array}
\right]$ the corresponding Killing field on $G$ is $X_3 = X_3^i L_i$ where 
\[ X_3^1 =  0 \]
\[ X_3^2 = e^{-ax_1}\left[-2x_3\cos(bx_1) +2x_2\sin(bx_1) \right] \]
\[X_3^3 = e^{-ax_1}\left[2x_3\sin(bx_1) +2x_2\cos(bx_1) \right] \]

One finds $X_3\arrowvert_e = 0$ and that 
\[P(X_3) = \left[ \begin{array}{*{20}c}
0  &  0   &  0 \\
0  &  0  &  -2  \\
0  &  2  &  0
\end{array} \right] . \] 

Then if $Y_3 = X_3 - 2R_1$ we find $P(Y_3) = 0$.  So $Y_3$ is parallel at $e$ and $Y_3|_e =-2e_1 \in S_e$.

Then since the $Y_i|e$ for $i=1,2,3$ are in $S_e$ and span all of $\mathfrak{g}$, we must have $S_e = \mathfrak{g} $ and $i(G) = 3$ as claimed.

\section{The unimodular nongeneric cases.}

In this section we determine $i(G)$ for the three unimodular nongeneric cases.  The results are
\begin{proposition}
Let $G$ be a simply-connected Lie group with left-invariant metric and corresponding metric Lie algebra $\mathfrak{g}$.  Then

(1) if $\mathfrak{g} = \mathfrak{e}(1,1)$ with triple $a = 0$ and $b=c$, then $i(G) = 1$.

(2) if $\mathfrak{g}$ is the Heisenberg algebra with triple $a=c=0$ and $b>0$, then $i(G) = 1$.

(3) if $\mathfrak{g}= \mathfrak{e}(2)$ with triple $a=0$ and $b+c=0$, then $i(G) = 3$.	
\end{proposition}

(1) For $\mathfrak{e}(1,1)$ with triple $a=0,\,b=c$, the Ricci matrix is already diagonalized for the standard coordinates with $d_i = e_i$:
\begin{equation}
R = R_d=  \left[ {\begin{array}{*{20}c}
	R_{11} & 0 & 0  \\
	0 & R_{22} & 0  \\
	0 & 0 & R_{33} \\
	\end{array} } \right]. \notag
\end{equation} 
where
\[R_{11} = -2b^2 \]
\[R_{22} = R_{33}= 0 .\]

As in section 5, any one parameter group of isometries preserving $e$ gives a family of orthogonal transformations $a(t)$ which commute with $R = R_d$.  Then $a(t)$ must leave $e_1$ invariant while rotating the $e_2-e_3$ plane.  So $a(t)'\mid_{t=0} = \left[ \begin{array}{*{20}c}
0  &  0  &  0 \\
0  &  0  &  \lambda \\
0  &  -\lambda  &  0
\end{array} \right]$ for some constant $\lambda$.  Then by proposition (\ref{P:Xvanate}),
\[P(X) = a(t)'\mid_{t=0} = \left[ \begin{array}{*{20}c}
0  &  0  &  0 \\
0  &  0  &  \lambda \\
0  &  -\lambda  &  0
\end{array} \right] \] 
for the Killing field $X$ associated with the one parameter group.  Then  for any Killing field $Y = \alpha_i R_i  + \beta X$ we find (using the formulas (\ref{R(k,i)}) and recalling that $a=0,\,b=c > 0$) that 
\[P(Y) = \left[ \begin{array}{*{20}c}
0  &  \alpha_3 b  &  \alpha_2 b \\
-\alpha_3 b  &  0  &  \beta\lambda \\
-\alpha_2 b  &  - \beta\lambda  &  0
\end{array} \right].\]  So $Y$ can be parallel at $e$ only if $\alpha_2 = \alpha_3 = \beta = 0$. But then $Y = \alpha_1 R_1$ and $Y\lvert_e = \alpha_1 e_1$.  Then $S_e = \bar{S}_e$ and $i(G) = \bar{i}(G) = 1$ as claimed.

(2) For the Heisenberg algebra with triple $a=c=0$, $b>0$,  the Ricci matrix is already diagonalized for the standard coordinates with $d_i = e_i$:
\begin{equation} 
	R = R_d=  \left[ {\begin{array}{*{20}c}
			R_{11} & 0 & 0  \\
			0 & R_{22} & 0  \\
			0 & 0 & R_{33} \\
	\end{array} } \right].  \notag
\end{equation} 
where
\[R_{11} = R_{22} = -b^2/2 \]
\[R_{33}= b^2/2 .\]

As in section 5, any one parameter group of isometries preserving $e$ gives a family of orthogonal transformations $a(t)$ which commute with $R = R_d$.  Then $a(t)$ must leave $e_3$ invariant while rotating the $e_1-e_2$ plane.  So $a(t)'\mid_{t=0} = \left[ \begin{array}{*{20}c}
0  &  \lambda  &  0 \\
-\lambda  &  0  &  0 \\
0  & 0  &  0
\end{array} \right]$ for some constant $\lambda$. Then by proposition (\ref{P:Xvanate}),
\[P(X) = a(t)'\mid_{t=0} = \left[ \begin{array}{*{20}c}
0  &  \lambda  &  0 \\
-\lambda  &  0  &  0 \\
0  &  0  &  0
\end{array} \right] \] 
for the Killing field $X$ associated with the one parameter group. Then  for any Killing field $Y = \alpha_i R_i  + \beta X$ we find (using the formulas (\ref{R(k,i)}) and recalling that $a=c =0,\,b > 0$) that 

\[P(Y) = \left[ \begin{array}{*{20}c}
0  &  \alpha_3 b/2 + \beta\lambda &   \alpha_2 b/2 \\
-\alpha_3 b/2 - \beta\lambda  &  0  &  -\alpha_1 b/2 \\
-\alpha_2 b/2  &  \alpha_1 b/2  &  0
\end{array} \right].\]  So $Y$ can be parallel at $e$ only if $\alpha_1 = \alpha_2 = 0$ and $\alpha_3b/2 +\beta\lambda = 0$.  Then $Y = \alpha_3 R_3 + \beta X$ and $Y\lvert_e = \alpha_3 e_3$.  So $S_e = \text{span}(e_3)$ is one dimensional if there exists any such Killing field $X$ vanishing at $e = (0,0,0)$.  In section 12 we show that the vector field $X = x_2 L_1 -x_1 L_2 + \frac{b}{2} (x_1^2 + x_2^2)L_3$ is in fact a Killing field which vanishes at $e$. So $S_e = \text{span}(e_3)$ and $i(G) = 1$ as claimed.

(3) Finally, for  $\mathfrak{g}= \mathfrak{e}(2)$ with triple $a=0$ and $b+c=0$, the Ricci operator $R$ vanishes completely for the standard coordinates.  In fact all sectional curvatures vanish and $G$ is isometric to Euclidean space.  (By the Killing-Hopf theorem, see \cite{Wolf}. The metric for the standard coordinates is just $dx_1^2 +dx_2^2 +dx_3^2$.) 

The right invariant fields $R_2$ and $R_3$ are already  parallel at $e$ in this case (see equations (\ref{R(k,i)})), so $e_2,e_3 \in S_e$.  For $R_1$ we have $P(R_1) =  \left[ \begin{array}{*{20}c}
0  &  0  &  0 \\
0  &  0  &  -b \\
0  &  b  &  0
\end{array} \right]$.

The vector field 
\[X = \left[x_3\cos(bx_1) -x_2\sin(bx_1)\right] L_2 + \left[-x_3\sin(bx_1) -x_2\cos(bx_1)\right] L_3\] clearly vanishes at $e$ and in section 12 we check that $X$ is in fact a Killing field.  We also check in section 12 that $P(X) =  \left[ \begin{array}{*{20}c}
0  &  0  &  0 \\
0  &  0  &  1 \\
0  &  -1  &  0
\end{array} \right]$.  Then $R_1 + bX$ is a Killing field parallel at $e$  and $(R_1+bX)\lvert_e = e_1 \in S_e$.  So $S_e$ contains al1 $e_i$, $i = 1,2,3$, and $i(G)=3$ as claimed.

\section{Selected proofs.}

In this section we sketch the proofs of some of the results given above.

We begin by checking the formulas (\ref{R(k,i)}) for the matrices $P(R_s)$ corresponding to the right invariant fields $R_s$ of equations  (\ref{E:Right}).
\begin{proof}
We have   
\begin{multline} 
R_1=L_1+e^{-2ax_1}\left[m_{11}(x_1)(ax_2+cx_3)-\frac{c}{b}m_{21}(x_1)(bx_2+ax_3)     \right]L_2  \\
+e^{-2ax_1}\left[-m_{21}(x_1)(ax_2+cx_3)+  m_{11}(x_1)(bx_2+ax_3) \right]L_3 \notag
\end{multline} 
\begin{align} 
R_2 &= e^{-2ax_1}m_{11}(x_1)L_2-e^{-2ax_1}m_{21}(x_1)L_3  \notag \\
R_3 &=-e^{2ax_1}\frac{c}{b} m_{21}(x_1)L_2 + e^{-2ax_1}m_{11}(x_1)L_3. \notag
\end{align}

 We need to compute 
\[P(R_s) = D(R_s)^T\arrowvert_e + \Gamma(R_s) .\]
For $D(R_s)^T\arrowvert_e$, recall that we are evaluating derivatives at $e=(0,0,0)$ and that by (\ref{E:m}) we have $m_{11}(0) = 1$ and $m_{21}(0) = 0$.  Also recall that $M'(0)=A$.  We then find

\begin{gather} 
D(R_1)^T\arrowvert_e =  \left[ {\begin{array}{*{20}c}
	{0} & {0} & {0}	\\
	{0} & {a} & {c} \\
	{0} & {b} & {a}	
	\end{array} } \right]   \notag \\  
D(R_2)^T \arrowvert_e =  \left[ {\begin{array}{*{20}c}
	{0} & {0} & {0}	\\
	{-a} & {0} & {0} \\
	{-b} & {0} & {0}	
	\end{array} } \right]  \notag \\  
D(R_3)^T \arrowvert_e =  \left[ {\begin{array}{*{20}c}
	{0} & {0} & {0}	\\
	{-c} & {0} & {0} \\
    {-a} & {0} & {0}
	\end{array} } \right]  \notag  
\end{gather}

Next notice that $R_s^j\arrowvert_e = R_s^j(e)=1$ if $j=s$ and vanishes if $j \neq s$.  Then the formula (\ref{E:Xgamma}) gives 
\begin{gather} 
\Gamma(R_1)  =  \left[ {\begin{array}{*{20}c}
	{0} & {0} & {0}	\\
	{0} & {-a} & {-(b+c)/2} \\
	{0} & {-(b+c)/2} & {-a}	
	\end{array} } \right]   \notag \\  
\Gamma(R_2)  =  \left[ {\begin{array}{*{20}c}
	{0} & {a} & {(b+c)/2}	\\
	{0} & {0} & {0} \\
	{(b-c)/2} & {0} & {0}	
	\end{array} } \right]  \notag \\  
\Gamma(R_3)  =  \left[ {\begin{array}{*{20}c}
	{0} & {(b+c)/2} & {a}	\\
	{-(b-c)/2} & {0} & {0} \\
	{0} & {0} & {0}	
	\end{array} } \right]  \notag  
\end{gather}

Adding the pairs of matrices we get the formulas (\ref{R(k,i)}) for $R_s(k,i)$.
\end{proof}

We now sketch proofs that certain vector fields are Killing fields as claimed above.  We begin with a field from section 9. 
\begin{proof}
Recall that $a^2 = bc >0$ for this case. Then using the formula for $M(t)$ from section 3 when $c=a^2/b>0$ and $\lambda = a/b$ we find
\[m_{11}(t)=m_{22}(t)=e^{at}\cosh(at)=1/2+e^{2at}/2\] and \[m_{21}(t)=\frac{b}{c}m_{12}(t) =\frac{b}{a}e^{at}\sinh(at)= \frac{b}{a}(-1/2+e^{2at}/2).\]  Then (\ref{E:Left}) gives
\[L_1 = \partial/\partial x_1\] 
\[L_2 = (1/2 + e^{2ax_1}/2)\,\, \partial/\partial x_2 + \frac{b}{a}(-1/2 + e^{2ax_1}/2)\,\, \partial/\partial x_3\]     
\[L_3 = \frac{a}{b}(-1/2 + e^{2ax_1}/2)\,\,\partial/\partial x_2 + (1/2 + e^{2ax_1}/2 ) \,\, \partial/\partial x_3\,\,.\]

Consider the vector field $X=X^iL_i$ where
\[X^1 = bx_2 + ax_3\]
\[X^2 = \frac{1}{8b}\left[ e^{-2ax_1}Y  +e^{2ax_1}(a^2 - 3b^2)/a + 2(b^2 - a^2)/a  \right] \]
\[X^3 = \frac{1}{8a}\left[ e^{-2ax_1}Y  +e^{2ax_1}(b^2 - 3a^2)/a + 2(a^2 - b^2)/a      \right]\, \]
and $Y = 4a(bx_2 + ax_3)^2 + (a^2 + b^2)/a $.

The condition for $X$ to be a Killing field is 
\[L_j(X^k) + L_k(X^j) + X^i(c^j_{ki} + c^k_{ji}) = 0\]
for all $j,k$.
or
\[D(X)+D(X)^T+C(X)=0.\]
Using $D(X)_{jk}=L_jX^k$, a tedious but straightforward computation gives the matrix 
\[D(X) = \left[ {\begin{array}{*{20}c}
	{0} & {-aX^2-cX^3-be^{2ax_1}} & {-bX^2-aX^3-ae^{2ax_1}}	\\
	{be^{2ax_1}} & {aX^1} & {bX^1} \\
	{ae^{2ax_1}} & {cX^1} & {aX^1}	
	\end{array} } \right].\]
By (\ref{E:Xc}) we also have
\[C(X) = \left[ {\begin{array}{*{20}c}
	{0} & {aX^2+cX^3} & {bX^2+aX^3}	\\
	{aX^2+cX^3} & {-2aX^1} & {-(b+c)X^1} \\
	{bX^2+aX^3} & {-(b+c)X^1} & {-2aX^1}	
	\end{array} } \right]\]
It is then easy to check that 
\end{proof}

We next sketch a proof that the field $X_1$ from section 10 is a Killing field and that $P(X_1)  = \left[ \begin{array}{*{20}c}
0  &  -1/a   &  0 \\
1/a  &  0  &  0  \\
0  &  0  &  0
\end{array} \right] . $

\begin{proof}
For this case the matrix is $\left[\begin{array}{*{20}c}
a  &  -b \\
b  &  a
\end{array}  \right]$ and we have (see section 3)
\[M(t) = e^{At} = e^{at} 
\left[ 
{\begin{array}{*{20}c}
	\cos(bt) & -\sin(bt)  \\
	\sin(bt) & \cos(bt)
	\end{array} }
\right]    \]
Then
\[L_1 = \partial/\partial x_1\] 
\[L_2 = e^{ax_1}\cos(bx_1)\,\, \partial/\partial x_2 + e^{ax_1}\sin(bx_1)\,\, \partial/\partial x_3\]     
\[L_3 = -e^{ax_1}\sin(bx_1)\,\,\partial/\partial x_2 + 
 e^{ax_1}\cos(bx_1)\,\, \partial/\partial x_3\,\,.\]

We must show that the field $X_1=X_1^iL_i$ satisfies 
\[D(X)+D(X)^T+C(X)=0.\]
when
\[ X_1^1 = \dfrac{-2x_2}{a} \]
\[ X_1^2 = e^{-ax_1}\left[\cos(bx_1)(x_3^2 - x_2^2)  -\sin(bx_1)(2x_2x_3) \right] + \dfrac{e^{ax_1}\cos(bx_1)}{a^2}   \]
\[X_1^3 = e^{-ax_1}\left[-\sin(bx_1)(x_3^2 - x_2^2) -\cos(bx_1)(2x_2x_3) \right] - \dfrac{e^{ax_1}\sin(bx_1)}{a^2}      \]	
Another tedious but straightforward computation gives the matrix $D(X_1)$:
\[\left[ {\begin{array}{*{20}c}
	{0} & {-aX_1^2+bX_1^3+2e^{ax_1}\cos(bx_1)/a} & {-bX_1^2-aX_1^3-2e^{ax_1}\sin(bx_1)/a}	\\
	{-2e^{ax_1}\cos(bx_1)/a} & {-2x_2} & {-2x_3} \\
	{2e^{ax_1}\sin(bx_1)/a} & {2x_3} & {-2x_2}	
	\end{array} } \right]\]

Then using (\ref{E:Xc}) with $c=-b$ and $X_1^1=(-2x_2)/a$, we also have
\[C(X) =\left[ {\begin{array}{*{20}c}
	{0} & {aX_1^2-bX_1^3} & {bX_1^2+aX_1^3}	\\
	{aX_1^2-bX_1^3} & {4x_2} & {0} \\
	{bX_1^2+aX_1^3} & {0} & {4x_2}	
	\end{array} } \right].\]
It is then easy to check that $D(X)+D(X)^T+C(X)=0$, so $X_1$ is in fact a Killing field.

Next note that evaluating the matrix above at $e=(0,0,0)$ gives 
\[D(X_1) \arrowvert_e = \left[ {\begin{array}{*{20}c}
	{0} & {1/a} & {-b/a^2}	\\
	{-2/a} & {0} & {0} \\
	{0} & {0} & {0}	
	\end{array} } \right].\]

Then using $X_1^1(e) = X_1^3(e) = 0$, $X_1^2(e) = 1/a^2$, and $b+c=0$, the matrix (\ref{E:Xgamma}) gives 
\begin{equation}
\Gamma(X_1) =
\left[ {\begin{array}{*{20}c}
	{0} & {1/a} & {0}	\\
	{0} & {0} & {0} \\
	{b/a^2} & {0} & {0}	
	\end{array} } \right] \notag
\end{equation} 
 
Taking a transpose matrix and adding we find
\[P(X_1)  =D(X_1)^T\arrowvert_e +  \Gamma(X_1) = \left[ \begin{array}{*{20}c}
0  &  -1/a   &  0 \\
1/a  &  0  &  0  \\
0  &  0  &  0
\end{array} \right] \]
 as claimed in section 9.
\end{proof}

Finally we verify that the vector fields in section 11 are in fact Killing fields.

\begin{proof}
For the Heisenberg algebra with matrix $\left[\begin{array}{*{20}c}
0  &  0 \\
b  &  0
\end{array}  \right]$ we have (see section 4)
\[M(t) = e^{At} = 
\left[ 
{\begin{array}{*{20}c}
	1 & 0  \\
	bt & 1
	\end{array} }
\right]    \]
and then
\[L_1 = \partial/\partial x_1\] 
\[L_2 = \partial/\partial x_2 + bx_1 \partial/\partial x_3\]     
\[L_3 = \partial/\partial x_3\,\,.\]
 
The vector field $X = x_2 L_1 -x_1 L_2 + \frac{b}{2} (x_1^2 + x_2^2)L_3$ clearly vanishes at $e$ and we claim $X$ is in fact a Killing field.	
Recalling that $D(X)_{jk}=L_jX^k$, we compute the matrix
\[D(X) = \left[ {\begin{array}{*{20}c}
	{0} & {-1} & {bx_1}	\\
	{1} & {0} & {bx_2} \\
	{0} & {0} & {0}	
	\end{array} } \right].\]	
Then using (\ref{E:Xc}) with $a=c=0$, $X^1=x_2$, and $X^2=-x_1$ we also have
\[C(X) =\left[ {\begin{array}{*{20}c}
	{0} & {0} & {-bx^1}	\\
	{0} & {0} & {-bx_2} \\
	{-bx_1} & {-bx_2} & {0}	
	\end{array} } \right].\]	
It is then easy to check that 
$D(X)+D(X)^T+C(X)=0$, so $X$ is in fact a Killing field.
	
For $\mathfrak{g}=\mathfrak{e}(2)$ and the matrix $\left[\begin{array}{*{20}c}
0  &  -b \\
b  &  0
\end{array}  \right]$ we have (see section 4)
\[M(t) = e^{At} = 
\left[ 
{\begin{array}{*{20}c}
	\cos(bx_1) & -\sin(bx_1)  \\
	\sin(bx_1) & \cos(bx_1)
	\end{array} }
\right]    \]
and then
\[L_1 = \partial/\partial x_1\] 
\[L_2 = \cos(bx_1) \partial/\partial x_2 + \sin(bx_1) \partial/\partial x_3\]     
\[L_3 = -\sin(bx_1)\partial/\partial x_2 + \cos(bx_1)   \partial/\partial x_3\,\,.\]
The field 
\[X = \left[x_3\cos(bx_1) -x_2\sin(bx_1)\right] L_2 + \left[-x_3\sin(bx_1) -x_2\cos(bx_1)\right] L_3\]
clearly vanishes at $e$ and we claim it is a Killing field.
We compute the matrix
\[D(X) = \left[ {\begin{array}{*{20}c}
	{0} & {bX^3} & {-bX^2}	\\
	{0} & {0} & {-1} \\
	{0} & {1} & {0}	
	\end{array} } \right].\]	
Then using (\ref{E:Xc}) with $a=0$ and $b+c=0$ we also have
\[C(X) =\left[ {\begin{array}{*{20}c}
	{0} & {-bX^3} & {bX^2}	\\
	{-bX^3} & {0} & {0} \\
	{bX^2} & {0} & {0}	
	\end{array} } \right].\]	
It is then easy to check that 
$D(X)+D(X)^T+C(X=0 = 0$, so $X$ is in fact a Killing field.

Since $X$ vanishes at $e$, we have $\Gamma(X)=0$.  Evaluating partial derivatives at $e$, we find $D(X)\arrowvert_e =  \left[ \begin{array}{*{20}c}
0  &  0  &  0 \\
0  &  0  &  -1 \\
0  &  1  &  0
\end{array} \right]$.  Then 
\[P(X)=D(X)^T\arrowvert_e + \Gamma(X) = \left[ \begin{array}{*{20}c}
0  &  0  &  0 \\
0  &  0  &  1 \\
0  &  -1  &  0
\end{array} \right]\] as claimed in section 11.	
	
\end{proof}

\end{document}